\documentclass[a4paper,11pt]{amsart}

\usepackage[utf8]{inputenc}
\usepackage[dvips]{graphicx}
\usepackage[shortlabels]{enumitem}
\usepackage{amsmath,amssymb,color,enumitem,graphics,hyperref,latexsym,pgfplots,tikz,vmargin}
\usepackage{relsize}
\usepackage{dsfont}
\usepackage[mathscr]{euscript}

\hypersetup{
unicode=false,          
pdftoolbar=true,        
pdfmenubar=true,        
pdffitwindow=false,     
pdfstartview={FitH},    
pdftitle={Sets with large intersection properties in metric spaces},    
pdfauthor={Felipe Negreira, Emiliano Sequeira},     
pdfsubject={11J83, 28A78},   
pdfkeywords={Diophantine approximations, metric spaces, net measures}, 
pdfnewwindow=true,      
colorlinks=true,       
linkcolor=black,          
citecolor=black,        
filecolor=magenta,      
urlcolor=cyan           
}

\def\N{{\mathbb{N}}}

\def\R{{\mathbb{R}}}

\def\U{{\mathcal{U}}}

\newcommand{\Vol}{\operatorname{Vol}}
\newcommand{\Ker}{\operatorname{Ker}}
\newcommand{\I}{\operatorname{Im}}

\newtheorem{lemma}{Lemma}[section]

\newtheorem{theorem}[lemma]{Theorem}
\newtheorem{corollary}[lemma]{Corollary}

\theoremstyle{definition}

\theoremstyle{remark}
\newtheorem{remark}[lemma]{Remark}

\begin{document}

\title{De Rham's theorem for Orlicz cohomology}
\author[E. Sequeira]{Emiliano Sequeira}
\address{Sobolev Institute of Mathematics, 4 Acad. Koptyug Ave., Novosibirsk 630090, Russia}
\email{esequeiramanzino@gmail.com}


\begin{abstract}
We prove that the de Rham $L^\phi$-cohomology of a Riemannian manifold $M$ admiting a convenient triangulation $X$ is isomorphic to the simplicial $\ell^\phi$-cohomology of $X$ for any Young function $\phi$. This result implies the quasi-isometry invariance of the first one.
\end{abstract}

\maketitle

\section{Introduction}

Orlicz cohomology has been studied recently (see \cite{C,GK19,K17,K21,KP2,KP}) as a natural generalization of $L^p$-cohomology. These cohomology theories provide quasi-isometry invariants (and therefore have applications to classification problems, see for example \cite{C,Pa08,S}), and have connections with Sobolev and Poincaré type inequalities (\cite{GK19,GT16}) and harmonic functions (\cite{KP2,P}). 

There exist several definitions of  $L^p$-cohomology depending on the  context (simplicial complexes, Riemannian manifolds,  topological groups, general metric spaces), which are related by equivalence theorems (see \cite{BR,G,GKS88,Pa95,S}). The Orlicz case is not so explored yet in this aspect. We can mention, as results obtained in this direction, the equivalence between 
Orlicz cohomology of a star-bounded simplicial complex and the cohomology of the corresponding Orlicz-Sullivan complex (\cite{K21}), and the equivalence between the simplicial and de Rham Orlicz cohomology for Riemannian manifolds in the case of degree 1 (\cite{C}). In a previous work (\cite{S}) a generalization of this second result in all degrees was given in the case of Lie groups. Here we present an improvement of that. 

In the following theorem $L^\phi H^k(M)$ denotes the $k$-space of $L^\phi$-cohomology of the Riemannian manifold $M$ and $\ell^\phi H^k(X)$ the $k$-space of $\ell^\phi$-cohomology of the simplicial complex $X$, where $\phi$ is a Young function. The reduced versions of de Rham and simplicial Orlicz cohomology are denoted by $L^\phi \overline{H}^k(M)$ and $\ell^\phi \overline{H}^k(X)$. See $\S$\ref{OC} for the precise definitions. 

\begin{theorem}\label{main}
Let $M$ be a Riemannian manifold of dimension $n$ that admits a continuous triangulation $X$ such that:
\begin{enumerate}
    \item[(a)] there exists a uniform bound  for the degree of every vertex; and
    \item[(b)] every simplex of $X$ is biLipschitz homeomorphic to the standard one of the same dimension, where the Lipschitz constant does not depend on the simplex. 
\end{enumerate}
Then for any Young function $\phi$ and $k=0,\ldots, n$, the topological vector spaces $L^\phi H^k(M)$ and $\ell^\phi H^k(X)$ are isomorphic, and so are the Banach spaces $L^\phi \overline{H}^k(M)$ and $\ell^\phi \overline{H}^k(X)$. 
\end{theorem}

The existence of a triangulation as in Theorem \ref{main} can be ensured by imposing certain geometric hypothesis on the manifold $M$, such as bounded curvature and positive injectivity radius (see \cite{A}). 

In \cite{C} it is stablished that, if there exists a quasi-isometry between two simplicial complexes with certains properties, then their (reduced) $\ell^\phi$- cohomology are isomorphic. Using this result we can directly deduce the following: 

\begin{corollary}
Let $M_1$ and $M_2$ be two uniformly contractible Riemannian manifolds admiting  triangulations as in Theorem \ref{main}. If they are quasi-isometric, then for every Young function $\phi$ and $k=0,\ldots,n$ the topological vector spaces $L^\phi H^k(M_1)$ and $L^\phi H^k(M_2)$ are isomorphic, and so are the Banach spaces $L^\phi \overline{H}^k(M_1)$ and $L^\phi \overline{H}^k(M_2)$. 
\end{corollary}

A metric space $X$ is \textit{uniformly contractible} if there exists an increasing function $\varphi:[0,+\infty)\to [0,+\infty)$ such that for every $x\in X$ and $r>0$ the ball $B(x,r)=\{y\in X:|x-y|< r\}$ is contractible in $B(x,\varphi(r))$.

\section{Preliminaries}

\subsection{Cochain complexes and bicomplexes}

A \textit{cochain complex} (or simply \textit{complex} from here on) will be a sequence of vector topological spaces $C^*=\left\{C^k\right\}_{k\in\N}$ together with a sequence of continuous linear maps $d=d_k:C^k\to C^{k+1}$ such that $d\circ d=0$. We denote it by $(C^*,d_*)$ or $(C^*,d)$ and say that $d$ is the \textit{derivative} of the complex. 

The \textit{cohomology} of the complex $(C^*,d)$ is the family of vector topological spaces 
$$H^k=\frac{\Ker d_k}{\I d_{k-1}},$$
where $d_{-1}=0$.

If $(C^*,d)$ is a complex of Banach spaces, we also consider its reduced cohomology by the family of Banach spaces (or more generally, Frechét spaces)
$$\overline{H}^k=\frac{\Ker d_k}{\overline{\I d_{k-1}}}.$$

A \textit{cochain map} between two cochain complexes $(C^*,d_*)$ and $(D^*,\delta_*)$ is a family of continuous linear maps $f_k:C_k\to D_k$ that commutes with the derivatives, i.e. $\delta_k\circ f_k=f_{k+1}\circ d_k$ for every $k$. A cochain map naturally defines a continuous linear map between the corresponding (relative) cohomology spaces.

Two cochain maps $f_*,g_*:(C^*,d_*)\to (D^*,\delta_*)$ are \textit{homotopic} if there exists a family of continuous linear maps $p_k:C^k\to D^{k-1}$ such that 
\begin{equation}
\left\{\begin{array}{cc}
    p_{k+1}\circ d_k+\delta_{k-1}\circ p_k= f_k - g_k &\text{ if }k\geq 1 \\
    p_1\circ d_0 = f_0-g_0 &.
\end{array}\right.    
\end{equation}
This implies that $f_*$ and $g_*$ define the same map in cohomology. We say that two complexes $(C^*,d_*)$ and $(D^*,\delta_*)$ are \textit{homotopically equivalent} if there exist cochain maps $f_*:C^*\to D^*$, $g_*:D^*\to C^*$  such that $f_*\circ g_*$ and  $g_*\circ f_*$ are homotopic to the identity. If $D^*\subset C^*$, $d_*=\delta_*$ and $g_*$ is the inclusion we said that $(C^*,d_*)$ retracts to $(D^*,\delta_*)$. 

If two complexes are homotopically equivalent, then their cohomologies are isomorphic (in the sense of vector topological spaces). If they are in addition complexes of Banach (or Fréchet) spaces, then their reduced cohomologies are also isomorphic.

\medskip

The method we will use to prove Theorem \ref{main} is a bicomplex argument as the used, for example, in \cite[Theorem 8.9]{BT}, \cite{G} and \cite{Pa95}.

By a \textit{bicomplex} we mean a family of topological vector spaces $\{C^{k,m}\}_{k,m\in\N}$ together with linear continuous maps $d':C^{k,m}\to C^{k+1,m}$ and $d'':C^{k,m}\to C^{k,m+1}$ sucht that $d'\circ d'=0$, $d''\circ d''=0$ and $d'\circ d''+d''\circ d'=0$. It will be denoted by $(C^{*,*},d',d'')$. The \textit{rows} and \textit{colums} of the bicomplex are the complexes of the form $(C^{*,m},d')$ and $(C^{k,*},d'')$ respectively.

Observe that, if $(C^{*,*},d',d'')$ is a bicomplex, then one can define the \textit{kernels horizontal} and \textit{vertical} complexes $(E^*,d')$ and $(F^*,d'')$, where $E^k=\Ker d''|_{C^{k,0}}$ and $F^m=\Ker d'|_{C^{0,m}}$. Observe that the condition $d'\circ d''+d''\circ d'=0$ implies that they are cochain complexes.

And important fact about bicomplexes that we will use in the proof of Theorem \ref{main} is the following result, whose proof can be found in \cite{Pa95,G,S}.

\begin{theorem}\label{bicomplejo}
Let $(C^{*,*},d',d'')$ a bicomplex. Suppose that every row $(C^{*,m},d')$ retracts to $(F^0\to 0\to 0 \to \cdots)$ and every colum $(C^{k,*},d'')$ retracts to $(E^0\to 0\to 0\to \cdots)$. Then the kernels horizontal and vertical complexes of the bicomplex are homotopically equivalent. 
\end{theorem}

\subsection{Young functions and Orlicz spaces}

By a \textit{Young function} we mean a pair and convex function $\phi:\R\to [0,+\infty)$ that satisfy $\phi(t)=0$ if, and only if, $t=0$.

If $(Z,\mu)$ is a measure space and $f:Z\to\R$ is a measurable function, we define
$$\rho_\phi(f)=\int_Z \phi\left(f(x)\right)d\mu(x).$$
The Orlicz space of $(Z,\mu)$ associated to $\phi$ is the Banach space $L^\phi(Z)=L^\phi(Z,\mu)$ of classes of functions $f:Z\to\R$ for which there exists a constant $\alpha>0$ such that
$\rho_\phi(f/\alpha)<+\infty$, equipped with the Luxemburg norm
$$\|f\|_{L^\phi}=\inf\left\{\alpha>0 : \rho_\phi\left(\frac{f}{\alpha}\right)d\mu\leq 1\right\}.$$
If $Z$ is countable and $\mu$ is the counting measure, we write $L^\phi(Z)=\ell^\phi(Z)$.

Observe that, if we consider the Young function $\phi_p(t):=|t|^p$ with $p\geq 1$, then $L^\phi(Z)=L^p(Z)$. Other examples of Young functions are
$$\phi_{p,\kappa}(t)=\frac{|t|^p}{\log(e+|t|^{-1})},\ (p\geq 1, \kappa\geq 0),$$
which are used in \cite{C} for studying the large scale geometry of Heintze groups. These function have some nice properties: they are $N$-functions with the $\Delta_2$-condition if $(p,\kappa)\neq (1,0)$ (see definitions in \cite[Chapter II]{RR}). An example of Young function that does not satisfy the previous conditions can be $\phi(t)=e^{|t|}-1$.

\begin{remark}\label{ObsEqivalencia}
If $\lambda\geq 1$ is any constant, then $\rho_{\lambda\phi}(f)=\lambda\rho_{\phi}(f)\geq \rho_\phi(f)$, which implies $\|f\|_{L^\phi}\leq \|f\|_{L^{\lambda\phi}}$. Furthermore, the convexity of $\phi$ implies $\rho_{\lambda\phi}(f)\leq \rho_\phi(\lambda f)$, hence $\|f\|_{L^{\lambda\phi}}\leq \lambda\|f\|_{L^\phi}$. We conclude that for every $\lambda>0$ the spaces $L^\phi(Z)$ and $L^{\lambda\phi}(Z)$ are isomorphic, and
$$K^{-1}\|\ \|_{L^\phi}\leq \|\ \|_{L^{\lambda\phi}}\leq K \|\ \|_{L^\phi},$$
where $K=\max\{\lambda,\lambda^{-1}\}$.
\end{remark}

For more details about Orlicz spaces we refer to \cite{RR}.

\subsection{Orlicz cohomology}\label{OC}

We say that a simplicial complex $X$ equipped with a length distance has \textit{bounded geometry} if it has finite dimension and there exist a constant $C>0$ and a function $N:[0,+\infty)\to\N$ such that
\begin{enumerate}
    \item[(c)] the diameter of every simplex is bounded by $C$; and
    \item[(d)] for every $r\geq 0$, the number of simplices contained in a ball of radius $r$ is bounded by $N(r)$.
\end{enumerate}
Observe that the triangulation in Theorem \ref{main} (satisfying conditions (a) and (b)) has bounded geometry.

Denote by $X_k$ the set of $k$-simplices in $X$ and consider the cochain complex $(\ell^\phi(X_*),\delta)$,
with the usual coboundary operator $\delta$ defined by $\delta\theta(\sigma)=\theta(\partial\sigma)$, where if $\Delta=(v_0,\ldots,v_k)\in X_k$,
$$\partial\Delta=\sum_{i=0}^k \partial_i\Delta=\sum_{i=0}^k(v_0,\ldots,\hat{v_i},\ldots,v_k).$$
It is not difficult to see, using bounded geometry, that $\delta:\ell^\phi(X_k)\to \ell^\phi(X_{k+1})$ is welld-defined and continuous for every $k$. 
The cohomology of this complex is called the \textit{simplicial Orlicz cohomology} of $X$ associated to the Young function $\phi$ (or, more simply, the $\ell^\phi$\textit{-cohomology} of $X$) and denoted by $\ell^\phi H^*(X)$. Since $(\ell^\phi(X_*),\delta)$ is a complex of Banach spaces, we also consider the \textit{reduced simplicial Orliz cohomology} of $X$ associated to $\phi$ (or \textit{reduced $\ell^\phi$-cohomology} of $X$) as its reduced cohomology. We denote it by $\ell^\phi \overline{H}^*(X)$.

\medskip

Now take a Riemannian manifold $M$ and denote by $\Omega^k(M)$ the space of all (smooth) differential $k$-forms on $M$. Consider
$$L^\phi\Omega^k(M)=\{\omega\in\Omega^k(M) : \|\omega\|_{L^\phi},\|d\omega\|_{L^\phi}<+\infty\},$$ 
equipped with the norm $|\omega|_{L^\phi}=\|\omega\|_{L^\phi}+\|d\omega\|_{L^\phi}$. Here $d$ denotes the usual exterior derivative and $\|\omega\|_{L^\phi}$ is the Luxemburg norm of the function
$$x\mapsto|\omega|_x=\sup\left\{|\omega_x(u_1,\ldots,u_k)| : u_i\in T_xM \text{ for }i=1,\ldots,k, \text{ with } \|u_i\|_x=1\right\}$$
in the measure space $(M,dV)$, where $\|\ \|_x$ is the Riemannian norm on the tangent space $T_xM$ and $dV$ is the Riemannian volume on $M$. 

We can consider the cohomology of the complex $(L^\phi\Omega^k(M),d)$, which we call the \textit{smooth Orlicz-de Rham cohomology} of $M$ associated to $\phi$ (or \textit{smooth $L^\phi$-cohomology} of $M$) and denote it by $L^\phi H^{*}_{s}(M)$. 

Since $(L^\phi \Omega^*(M),d)$ is not a complex of Banach we take $(L^\phi C^*(M),d)$ the complex of their  completions  where $d$ is the continuous extension of the exterior derivative. From this we obtain the \textit{Orlicz-de Rham cohomology} of $M$ associated with $\phi$ (or $L^\phi$\textit{-cohomology} of $M$), denoted by $L^\phi H^*(M)$ 
and the \textit{reduced Orlicz-de Rham cohomology} of $M$ (or \textit{reduced} $L^\phi$\textit{-cohomology} of $M$), denoted by $L^\phi\overline{H}^*(M)$. 

In \cite{KP} it is prove that, under certain hypothesis on $\phi$, the Orlicz-de Rham cohomology is isomorphic to the smooth version. We give (at the end of this work) a proof of that for any Young function under some hypotesis on the manifold $M$.

\begin{remark}\label{ObsDef}
A \textit{measurable $k$-form} on $M$ is a function $x\mapsto\omega_x$, where $\omega_x$ is an alternating $k$-linear form on $T_xM$, such that the coefficients of $\omega$ for every parmetrization of $M$ are all measurable. We denote by $L^\phi(M,\Lambda^k)$ the space of $L^\phi$-integrable measurable $k$-forms up to almost everywhere zero forms. It is a Banach space equipped with the Luxemburg norm $\|\ \|_{L^\phi}$. 

Since $L^\phi\Omega^k(M)\subset L^\phi(M,\Lambda^k)$ and the inclusion is continuous, one can prove using Hölder's inequality ($\|fg\|_{L^1}\leq 2\|f\|_{L^\phi}\|g\|_{L^{\phi^*}}$, where $\phi^*$ is complementary Young function of $\phi$, see \cite[Section 3.3]{RR}) that $L^\phi C^k(M)$ can be seen as a space of measurable $k$-forms in $L^\phi(M,\Lambda^k)$ that have weak derivatives in $L^\phi(M,\Lambda^{k+1})$. 

We say that a measurable $(k+1)$-form $\varpi$ is the weak derivative of $\omega\in L^\phi(M,\Lambda^k)$ if for every differential $(n-k-1)$-form with compact support $\beta$ one has 
$$\int_M \varpi\wedge\beta=(-1)^{k+1}\int_M \omega\wedge d\beta.$$
\end{remark}

Throughout this work we will say $k$-\textit{form} to mean a measurable $k$-form and \textit{differential $k$-form} to mean a smooth one. 

\section{Poincare's lemma for Orlicz cohomology}

The following result is previously proved in \cite{GK19}. We present a proof based on Lemma 8 of \cite{Pa95} in order to show a construction that will be used in the proof of Theorem \ref{main}.

\begin{lemma}\label{Poincare}
Let $B=B(0,1)$ be the open unit ball in the Euclidean space $\R^n$.
Then both complexes $(L^\phi \Omega^*(B),d)$ and $(L^\phi C^*(B),d)$ retract to the complex $(\R\to 0\to 0\to \ldots)$.
\end{lemma}

\begin{proof}
For a fixed $x\in B$ we consider $\varphi_x:[0,1]\times B\to B$, $\varphi_x(t,y)=ty+(1-t)x$ and $\eta_t:B\to [0,1]\times B$, $\eta_t(y)=(t,y)$. We denote $\frac{\partial}{\partial t}=(1,0)\in [0,1]\times B$ and define the $\frac{\partial}{\partial t}$-contraction of a $k$-form $\omega$ as the $(k-1)$-form 
$$\iota_{\frac{\partial}{\partial t}}\omega_y (u_1,\ldots,u_{k-1})=\omega_y\left(\mathsmaller{\frac{\partial}{\partial t}},u_1,\ldots,u_{k-1}\right).$$

Now we define, for $y\in B$, $\omega\in L^\phi\Omega^k(B)$ and $u_1,\ldots,u_{k-1}\in T_yB=\R^n$,
\begin{align*}
\chi_x(\omega)_y(u_1,\ldots,u_{k-1})&=\left(\int_0^1 \eta_t^*\left(\iota_{\frac{\partial}{\partial t}}\varphi_x^*\omega\right)dt\right)_y(u_1,\ldots,u_{k-1})\\   &=\int_0^1 \eta_t^*\left(\iota_{\frac{\partial}{\partial t}}\varphi_x^*\omega\right)_y(u_1,\ldots,u_{k-1})dt,
\end{align*}
where $\eta_t^*$ and $\varphi^*$ are the classical pull-back transformations of the respective functions. 

Observe that the coefficients of 
$\eta_t^*\left(\iota_{\frac{\partial}{\partial t}}\varphi_x^*\omega\right)_y$ are smooth on $t,x$ and $y$. Thus, using the Leibniz integral rule, we can observe that $\chi_x(\omega)$ is a differential $(k-1)$-form.

For every $(k-1)$-simplex $\sigma$ we have
\begin{align*}
\int_\sigma \chi_x(\omega) 
= \int_\sigma\int_0^1 \eta_s^*(\iota_{\frac{\partial}{\partial t}}\varphi_x^*\omega)ds
= \int_{[0,1]\times\sigma}\varphi_x^*\omega
= \int_{\varphi_x([0,1]\times\sigma)}\omega
= \int_{C_\sigma}\omega,
\end{align*}
where $C_\sigma$ is the cone centered at $x$, defined as follows: If $\sigma=(x_0,\ldots,x_{k-1})$, then $C_\sigma=(x,x_0,\ldots,x_{k-1})$.

Suppose that $\sigma$ is a $k$-simplex in $B$ with $\partial \sigma=\tau_0+\cdots+\tau_k$ and $\omega\in \Omega^k(B)$. Then, using Stoke's theorem, we have
\begin{align*}
\int_\sigma \chi_x(d\omega) = \int_{C_\sigma} d\omega = \int_{\partial C_\sigma}\omega = \int_\sigma \omega-\sum_{i=0}^k \int_{C_{\tau_i}}\omega =\int_\sigma \omega- \int_{\partial\sigma}\chi_x(\omega)=\int_\sigma \omega- \int_{\sigma}d\chi_x(\omega).   
\end{align*}
Since this is true for every $k$-simplex, we conclude that 
\begin{equation}\label{dos}
\chi_x d+d\chi_x=Id    
\end{equation}
(see for example \cite[Chapter IV]{W}). Observe that if $\omega$ is closed, then $\chi_x(\omega)$ is a primitive of $\omega$, so it is enough to prove the classic Poincaré's lemma. However, in our case we need a $L^\phi$-primitive, so we take a convenient average. 

Define
$$h(\omega)=\frac{1}{\Vol\left(\frac{1}{2}B\right)}\int_{\frac{1}{2}B}\chi_x(\omega)dx,$$
where $\frac{1}{2}B=B\left(0,\frac{1}{2}\right)$.

For the same argument as before $h(\omega)$ is smooth and
\begin{equation*}
dh(\omega)=\frac{1}{\Vol\left(\frac{1}{2}B\right)}\int_{\frac{1}{2}B}d\chi_x(\omega)dx.
\end{equation*}
Here is important the fact that we are integrating on a ball whose closure is strictly contained in $B$. Therefore, using (\ref{dos}) we have
\begin{equation}\label{homot}
dh(\omega)+h(d\omega)=\omega
\end{equation}
for all $\omega\in\Omega^k(B)$ with $k\geq 1$.

We have to prove that $h$ is well-defined and continuous from $L^\phi\Omega^k(B)$ to $L^\phi\Omega^{k-1}(B)$. To this end we first bound $|\chi_x(\omega)|_y$ for $y\in B$ and $\omega\in\Omega^k(B)$. Since $\iota_{\frac{\partial}{\partial t}}\varphi^*\omega$ is a form on $[0,1]\times B$ that is zero in the direction of $\frac{\partial}{\partial t}$, we have 
$|\eta_t^*(\iota_{\frac{\partial}{\partial t}}\varphi^*\omega)|_y=|\iota_{\frac{\partial}{\partial t}}\varphi^*\omega|_{(t,y)}$ for every $t\in (0,1)$ and $y\in B$. After a direct calculation we get the estimate $$\left|\iota_{\frac{\partial}{\partial t}}\varphi^*\omega\right|_{(t,y)}\leq t^{k-1}|y-x||\omega|_{\varphi(t,y)}.$$ 
Hence, using the assumption that $t\in (0,1)$, we can write
\begin{equation}\label{uno}
|\chi(\omega)|_y\leq \int_0^1|y-x||\omega|_{\varphi(t,y)}dt.
\end{equation}

Consider the function $u:\R^n\to\R$ defined by $u(z)=|\omega|_z$ if $z\in B$ and $u(z)=0$ in the other case. Using (\ref{uno}) and the change of variables $z=ty+(1-t)x$  we have
\begin{align*}
\Vol\left(\mathsmaller{\frac{1}{2}}B\right)|h(\omega)|_y		
&\leq \int_{B\left(ty,\frac{1-t}{2}\right)}\int_0^1|z-y|u(z)(1-t)^{-n-1}dtdz \\
&=  \int_{B(y,2)}|z-y|u(z)\left(\int_0^1 \mathds{1}_{B\left(ty,\frac{1-t}{2}\right)}(z)(1-t)^{-n-1} dt\right)dz. 	
\end{align*}
Observe that $\mathds{1}_{B\left(ty,\frac{1-t}{2}\right)}(z)=1$ implies that $|z-y|\leq 2(1-t)$. Therefore
$$\int_0^1\mathds{1}_{B\left(ty,\frac{1-t}{2}\right)}(z)(1-t)^{-n-1}dt\leq \int_0^{1-\frac{1}{2}|z-y|}(1-t)^{-n-1}dt= \int_{\frac{1}{2}|z-y|}^1 r^{-n-1}dr\preceq \frac{1}{|z-y|^n}.$$
This implies
$$Vol(\mathsmaller{\frac{1}{2}}B)|h(\omega)|_y\preceq \int_{B(y,2)}|z-y|^{1-n}u(z)dz,$$
where $f\preceq g$ means $f\leq \text{const.} g$. (We also write $f\asymp g$ if $f\preceq g$ and $g\preceq f$.)

Using this estimate we have
\begin{align*}
\|h(\omega)\|_{L^\phi}
    &\preceq \inf\left\{\gamma>0 : \int_B\phi\left(\int_{B(y,2)}|z-y|^{1-n}\frac{u(z)}{\alpha}dz\right)dy\leq 1\right\}.
\end{align*}
Since $\int_{B(y,2)}|z-y|^{1-n}dz<+\infty$, we can use Jensen's inequality and write
\begin{align*}
\|h(\omega)\|_{L^\phi}
    &\preceq \inf\left\{\alpha>0 : \Vol(B(0,3))\int_B\int_{B(0,3)}\phi\left(\frac{u(z)}{\Vol(B(0,3))\alpha}\right)\frac{1}{|z-y|^{n-1}} dz dy\leq 1\right\}\\
    &=\inf\left\{\alpha>0 : \Vol(B(0,3))\int_{B(0,3)}\phi\left(\frac{|\omega|_z}{\Vol(B(0,3))\alpha}\right)\left(\int_B\frac{dy}{|z-y|^{n-1}}\right)dz\leq 1\right\}.
\end{align*}
We have that there exists a constant $K>0$ such that $\int_B\frac{dy}{|z-y|^{n-1}}\leq K$ for all $z\in B(0,3)$, thus
$$\|h(\omega)\|_{L^\phi}\preceq \Vol(B(0,3))\|\omega\|_{L^{\tilde{K}\phi}}\preceq \|\omega\|_{L^\phi},$$
where $\tilde{K}=K\Vol(B(0,3))$. 

By the identity $dh(\omega)=\omega-h(d\omega)$ we also have 
$$\|dh(\omega)\|_{\phi}\leq \|\omega\|_{\phi}+\|h(d\omega)\|_{\phi}\preceq \|\omega\|_{\phi}+\|d\omega\|_{\phi}.$$
We conclude that $h$ is well-defined and bounded for the norm $|\ \ |_{L^\phi}$ in all degrees $k\geq 1$.

If $\omega=df$ for certain function $f$ we observe that
$$\eta_t^*(\iota_{\frac{\partial}{\partial t}}\varphi^*_x df) (y)= df_{\varphi_x(t,y)}(y-x)=(f\circ\gamma)'(t),$$
where $\gamma$ is the curve $\gamma(t)=\varphi_x(t,y)$. Then $\chi_x(df)(y)=f(y)-f(x)$, from which we get
$$h(df)=f-\frac{1}{Vol\left(\mathsmaller{\frac{1}{2}}B\right)}\int_{\frac{1}{2}B} f.$$
We define $h:L^\phi\Omega^0(B)\to L^p\Omega^{-1}(B)=\R$ by
$h(f)=\frac{1}{Vol\left(\mathsmaller{\frac{1}{2}}B\right)}\int_{\frac{1}{2}B} f$,
which is clearly continuous because $\frac{1}{2}B$ has finite Lebesgue measure. Therefore 
\begin{equation}\label{homot2}
h(f)+h(df)=f.
\end{equation}
This shows that $(L^\phi\Omega^*,d)$ retracts to $(\R\to 0\to 0\to \cdots)$.

Note that, since $h$ is bounded, then it can be extended continuously to $L^\phi C^k(B)$ for every $k\geq 0$. Equalities \eqref{homot} and \eqref{homot2} are also true for every $\omega\in L^\phi C^k(B)$, which finishes the proof.
\end{proof}

\section{Proof of Theorem \ref{main}}

If $v$ is a vertex of $X$ we denote by $U_v$ the \textit{open star} of $v$, that is, the interior of the union of all simplices containing $v$. Observe that if $\Delta=(v_0,\ldots,v_m)\in  X_m$, then 
$$U_\Delta=U_{v_0}\cap\cdots\cap U_{v_m}$$
contains the interior of the simplex $\Delta$ as subset of $M$. By condition $(b)$ there exists a uniform constant $L\geq 1$ such that if $\Delta$ is a simplex of $X$, then $U_\Delta$ is $L$-biLipschitz homeomorphic to the open ball $B\subset\R^n$. Moreover, condition $(a)$ allows us to take a constant $N\in \N$ such that every point cannot belong to more than $N$ simplices of $X$. 

The triangulation $X$ is the nerve of the covering $\U=\left\{U_v : v\in X_0\right\}$. In fact, Theorem \ref{main} can be stated in terms of the existence of a convenient open covering.

For $k,m\geq 0$ we consider $C^{k,m}$ as the space of alternating functions
\begin{equation}\label{forma}
\omega=\prod_{\Delta\in X_m} \omega_\Delta,\ \omega_\Delta\in L^\phi C^k(U_\Delta).
\end{equation}
Where by \textit{alternating} we mean that for $i\neq j$, 
$$\omega_{(v_0,\ldots,v_i,\ldots,v_j,\ldots,v_m)}=-\omega_{(v_0,\ldots,v_j,\ldots,v_i,\ldots,v_m)}.$$ 

If $\omega\in C^{k,m}$, we consider
$$\varrho_\phi(\omega)=\sum_{\Delta\in X_m}\int_{U_\Delta}\phi\left(|\omega_\Delta|\right)dV,$$
where $dV$ is the volume on $M$. From this we take the Luxemburg norm
$$\|\omega\|_{C_\phi}=\inf\left\{\alpha>0 : \varrho_\phi\left(\frac{\omega}{\alpha}\right)\leq 1\right\}.$$

We define, for $\omega$ as in \eqref{forma}, 
$$d'\omega=(-1)^m\prod_{\Delta\in X_m}d\omega_\Delta,$$
and take
$$C_\phi^{k,m}=\left\{\omega\in C^{k,m} : \|\omega\|_{C_\phi},\|d'\omega\|_{C_\phi}<+\infty\right\},$$
which is naturally equipped with the norm $|\omega|_{C_\phi}=\|\omega\|_{C_\phi}+\|d'\omega\|_{C_\phi}$. Observe that $d': C_\phi^{k,m}\to  C_\phi^{k+1,m}$ is continuous and satisfy $d'\circ d'=0$.

For $\omega\in C_\phi^{k,m}$ and $\Delta\in X_{m+1}$ we put
$$(d''\omega)_{\Delta}=\sum_{i=0}^{m+1} (-1)^i \omega_{\partial_i\Delta}.$$

\begin{lemma}
$(C_\phi^{*,*},d',d'')$ is a bicomplex. 
\end{lemma}

\begin{proof}
One can directly verify that $d''\circ d''=0$ and $d'\circ d''+d''\circ d'=0$ so we need to prove that $d''$ is well-defined and continuous from $C_\phi^{k,m}$ to $C_\phi^{k,m+1}$. 

Let $\omega\in C_\phi^{k,m}$. It is clear that $d''\omega$ is alternating. Moreover, for every $\alpha>0$ we have,
\begin{align*}
\varrho_\phi\left(\frac{d''\omega}{\alpha}\right) &=\sum_{\Delta\in X_{m+1}}\int_{U_\Delta}\phi\left(\frac{1}{\alpha}\left|\sum_{i=0}^{m+1}(-1)^i\omega_{\partial_i\Delta}\right|\right) dV\\
& \leq \frac{1}{m+2} \sum_{\Delta\in X_{m+1}}\sum_{i=0}^{m+1} \int_{U_\Delta} \phi\left(\frac{m+2}{\alpha}|\omega_{\partial_i\Delta}|\right)dV\\
& \leq \frac{N}{m+2} \sum_{\Delta'\in X_m}\int_{U_{\Delta'}} \phi\left(\frac{m+2}{\alpha}|\omega_{\Delta'}|\right)dV\\
&=\varrho_{\frac{N}{m+2}\phi}\left(\frac{m+2}{\alpha}\omega\right).
\end{align*}
Where in the second line we use Jensen's inequality and in the third we use the fact that every $m$-simplex of $X$ can be in the boundary of at most $N$ $(m+1)$-simplices. This estimate and an argument as in Remark \ref{ObsEqivalencia} allow to conclude that $\|d''\omega\|_{C_\phi}\preceq \|\omega\|_{C_\phi}$. Using the identity $d'\circ d''+d''\circ d'=0$ we also have $\|d'd''\omega\|_{C_\phi}=\|d''d'\omega\|_{C_\phi}\preceq \|d''\omega\|_{C_\phi}$, thus $d''\omega\in C_\phi^{k,m+1}$ and $d''$ is continuous for the norm $|\ \ |_{C_\phi}$.
\end{proof}

We write $E^k_\phi=\Ker d''|_{C_\phi^{k,0}}$ and $F^m_\phi=\Ker d'|_{C_\phi^{0,m}}$. These spaces conform the horizontal and vertical complexes of kernels of the bicomplex $(C_\phi^{*,*},d',d'')$.

\begin{lemma}\label{retraccion1}
For every $m$ the complex $(C_\phi^{*,m},d')$ retracts to $(F_\phi^m\to 0\to 0\to \cdots)$. 
\end{lemma}

\begin{proof}
We have to define a family of continuous linear maps $H^k:C_\phi^{k,m}\to C_\phi^{k-1,m}$ such that
\begin{equation}\label{H}
\left\{\begin{array}{cc}
    H^{k+1}\circ d'+d'\circ H^k= Id &\text{ if }k\geq 1 \\
    H^1\circ d' + inc\circ H^0 =Id, & 
\end{array}\right.    
\end{equation}
where $C_\phi^{-1,m}=F_\phi^m$.

Remember that for every $\Delta\in X_m$ there exists a $L$-biLipschitz homeomorphism $f_\Delta:B\to U_\Delta$. By Rademacher's theorem it is differentiable almost everywhere, so we can consider the pull-back by $f$ of a $k$-form $\omega$ defined almost everywhere and denoted, as usual, by $f^*_\Delta\omega$.
Since $f_\Delta$ is $L$-biLipschitz one can easily prove that
$$|f_\Delta^*\omega|_x\leq L^k |\omega|_{f_\Delta(x)}\text{ and }|J_x f_\Delta|\leq L^n,$$
where $J_xf_\Delta$ denotes the Jacobian at $x$ of $f_\Delta$. This implies that $f^*_\Delta$ is well-defined and continuous from $L^\phi C^k(U)$ to $L^\phi C^k(B)$. The same can be done for the inverse $f^{-1}_\Delta$.

By Lemma \ref{Poincare} there exists a family of maps $h=h_k:L^\phi C^k(B)\to L^\phi C^{k-1}(B)$ such that
\begin{equation}\label{hh}
\left\{\begin{array}{cc}
    h_{k+1}\circ d+d\circ h_k= Id &\text{ if }k\geq 1 \\
    h_1\circ d + inc\circ h_0 =Id, & 
\end{array}\right.  
\end{equation}
(where $L^\phi C^{-1}(B)=\R$), and
\begin{equation}\label{Estimacionh}
\int_B \phi\left(\frac{|h(\omega)|_y}{\alpha}\right)dy \leq C \int_{B}\phi\left(\frac{|\omega|_y}{Vol(B(0,3))\alpha}\right)dz
\end{equation}
with $C$ a positive constant.

Now we define $H^k:C_\phi^{k,m}\to C_\phi^{k-1,m}$ by
$$(H^k \omega)_\Delta =(-1)^m (f_\Delta^{-1})^*\circ h \circ f_\Delta^* \omega_\Delta.$$
Observe that \eqref{H} follows directly from \eqref{hh} and the fact that the pull-back commutes with $d$. Moreover
\begin{align*}
\varrho_\phi\left(\frac{H^k\omega}{\alpha}\right) & = \sum_{\Delta\in X_m}\int_{U_\Delta}\phi\left(\frac{1}{\alpha}|(f_\Delta^{-1})^*\circ h\circ f_\Delta^* \omega_\Delta|_x\right)dV(x) \\
&\leq \sum_{\Delta\in X_m}\int_{U_\Delta} \phi\left(\frac{L^k}{\alpha}|h\circ f_I^* \omega_\Delta|_{f_\Delta^{-1}(x)}\right)dV(x)\\
&\leq L^n\sum_{\Delta\in X_m}\int_B \phi\left(\frac{L^k}{\alpha}|h\circ f_\Delta^* \omega_\Delta|_y\right)dy\\
&\leq L^n C \sum_{\Delta\in X_m}\int_{B} \phi\left(\frac{L^k|f_\Delta^*\omega_I|_y}{Vol(B(0,3))\alpha}\right)dy,
\end{align*}
Using that  $|f_\Delta^*\omega_I|_y\leq L^k |\omega|_{f_\Delta(y)}$ we obtain
\begin{align*}
\varrho_\phi\left(\frac{H^k\omega}{\alpha}\right) & \leq L^{2n} C \sum_{\Delta\in X_m}\int_{U_\Delta} \phi\left(\frac{L^{2k}|\omega|_{x}}{Vol(B(0,3))\alpha}\right)dV(x)\leq \varrho_{L^{2n}C\phi}\left(\frac{L^{2k}\omega}{Vol(B(0,3))\alpha}\right).   
\end{align*}
From here we deduce that $\|H^k\omega\|_{C_\phi}\preceq \|\omega\|_{C_\phi}$. Applying this and  \eqref{H}, we obtain $\|d'H^k\omega\|_{C_\phi}\preceq \|\omega\|_{C_\phi}+\|d'\omega\|_{C_\phi}$, which shows that $H^k$ is continuous for the norm $|\ \ |_{C_\phi}$.
\end{proof}

\begin{lemma}\label{retraccion2}
For every $k$ the complex $(C_\phi^{k,*},d'')$ retracts to $(E_\phi^k\to 0\to 0\to \cdots)$. 
\end{lemma}

\begin{proof}
We have to construct a family of linear and continuous maps  $P^m:C_\phi^{k,m}\to C_\phi^{k,m-1}$ such that
\begin{equation}\label{IdentidadP}
\left\{\begin{array}{cc}
    P^{m+1}\circ d''+d''\circ P^m= Id &\text{ if }m\geq 1 \\
    P^1\circ d'' + inc\circ P^0 =Id. & 
\end{array}\right.
\end{equation}
Here we denote $C_\phi^{k,-1}=E_\phi^k$.

We consider a smooth partition of unity $\{\eta_v\}_{v\in X_0}$ for the covering $\U$. It can be taken such that $|d\eta_v|$ is uniformly bounded independently of $v\in X_0$ because of the properties of the covering. Then we define 
$$(P^m \omega)_{\Delta} = (-1)^m\sum_{v\in X_0} \eta_v  \omega_{\Delta v}\ \text{ if }m\geq 1;\text{ and}$$
$$(P^0\omega)_{v_0}=\sum_{v\in X_0} \eta_v \omega_v|_{U_{v_0}}.$$
Where $\Delta v=(v_0,\ldots,v_m,v)$ if $\Delta=(v_0,\ldots,v_m)$, and $\omega_{\Delta v}=0$ if $\Delta v\notin X_{m+1}$. Remember that in the previous expressions the sums do not have more than $N$ terms. It is clear that if $\omega\in C_\phi^{k,0}$, then $d''P^0\omega=0$.

If $\omega\in C_\phi^{k,m}$ for $m\geq 1$, we have
\begin{align*}
(P^{m+1}\circ d'' \omega)_\Delta &= (-1)^{m+1}\sum_{v\in X_0} \eta_v (d''\omega)_{\Delta v}=(-1)^{m+1}\sum_{v\in X_0} \eta_v \left( \sum_{i=0}^{m+1} (-1)^i \omega_{\partial_i(\Delta v)}\right)\\
&= (-1)^{m+1}\sum_{v\in X_0}  \sum_{i=0}^m (-1)^i \eta_v \omega_{(\partial_i\Delta)v} + \sum_{v\notin \Delta}\eta_v \omega_\Delta. 
\end{align*}
On the other hand
\begin{align*}
(d''\circ P^m \omega)_\Delta &= \sum_{i=0}^{m} (-1)^i (P^m \omega)_{\partial_i\Delta}= (-1)^m \sum_{i=0}^m \sum_{v\in X_0} (-1)^i \eta_v \omega_{(\partial_i\Delta)v}\\
&= (-1)^m \sum_{i=0}^m \sum_{v\in X_0} (-1)^i \eta_v \omega_{(\partial_i\Delta)} +  \sum_{v\in \Delta}  \eta_v \omega_{\Delta}
\end{align*}
This shows that for every $\Delta\in X_m$,
$$(P^m\circ d'' \omega)_\Delta+(d''\circ P^m \omega)_\Delta=\omega_\Delta.$$

Now suppose that $m=0$. If $x\in U_{v_0}$, we have
\begin{align*}
(P^1\circ d''\omega)_{v_0} &= -\sum_{v\in X_0} \eta_v (d''\omega)_{(v_0,v)}=-\sum_{v\in X_0} \eta_v \left(\omega_v-\omega_{v_0(x)}\right)\\ 
& = \sum_{v\in X_0}\eta_v\omega_{v_0}-\sum_{v\in X_0}\eta_v\omega_v =\omega_{v_0}-(inc\circ P^0 \omega)_{v_0}.
\end{align*}

To finish the proof let us show that $P^m$ is continuous (and that its image is where it must be). Take $\omega\in C_\phi^{k,m}$ for any $m\geq 0$, then, using Jensen's inequality and the fact that $\eta_v$ is supported in $U_v$ and is bounded by $1$, we obtain
$$\varrho_\phi\left(\frac{P^m \omega}{\alpha}\right)\leq \varrho_{(m+1)\phi}\left(\frac{N\omega}{\alpha}\right),$$
which allows to conclude that $\|P^m\omega\|_{C_\phi}\preceq \|\omega\|_{C_\phi}$. Moreover, if $m\geq 1$, we have 
$$|(d'\circ P^m\omega)_\Delta| \leq \sum_{v\in X_0} |d\eta_v| |\omega_{\Delta v}| + |P^m(d'\omega)_\Delta|.$$
We can apply the previous estimation in the second term and the fact that $|d\eta_v|$ is uniformly bounded independently from $v$ to get a  $\|d'\circ P^m\omega\|_{C_\phi}\preceq \|\omega\|_{C_\phi}+\|d'\omega\|_{C_\phi}$. An analogous estimate can be proved for $m=0$.  As a conclusion we have that $P^m\omega\in C_\phi^{k,m}$ and $P^m$ is continuous for the norm $|\ \ |_{C_\phi}$. 

\end{proof}

\begin{lemma}\label{horizontal}
There exists a family of isomorphisms $\mathcal{E}:E_\phi^*\to L^\phi C^*(M)$ such that $\mathcal{E}\circ d'=d\circ\mathcal{E}$.
\end{lemma}

\begin{proof}
If $\omega\in E_\phi^k$, then for every $v_0,v_1\in X_0$ we have $\omega_{v_0}|_{U_{v_1}}=\omega_{v_1}|_{U_{v_0}}$ (a.e.). This allows to define a measurable $k$-form $\tilde{\omega}=\mathcal{E} \omega$ such that $\omega_v=\tilde{\omega}|_{U_v}$ for every $v\in X_0$. On the other hand, the inverse of $\Psi$ is clearly defined by
$$(\mathcal{E}^{-1}\tilde{\omega})_v=\tilde{\omega}|_{U_v}.$$

To see that both maps $\mathcal{E}$ and $\mathcal{E}^{-1}$ are well-defined and continuous embeddings observe that 
\begin{align*}
\varrho_\phi\left(\frac{\omega}{\alpha}\right) &= \sum_{v\in X_0}\int_{U_v} \phi \left(\frac{|\omega_v|}{\alpha}\right)dV = \sum_{v\in X_0}\int_{U_v} \phi \left(\frac{|\tilde{\omega}|}{\alpha}\right)dV \leq (n+1) \int_M \phi \left(\frac{|\tilde{\omega}|}{\alpha}\right)dV\\
&\leq (n+1)\sum_{v\in X_0}\int_{U_v} \phi \left(\frac{|\tilde{\omega}|}{\alpha}\right)dV = (n+1)\varrho_\phi\left(\frac{\omega}{\alpha}\right)
\end{align*}
This shows that $\|\omega\|_{C_\phi}\asymp \|\tilde{\omega}\|_{L^\phi}$. It is clear that both $\mathcal{E}$ and $\mathcal{E}^{-1}$ commute with the derivatives, which, together with the previous  estimate imply $\|d'\omega\|_{C_\phi}\asymp \|d\tilde{\omega}\|_{L^\phi}$ and finish the proof. 

\end{proof}

\begin{lemma}\label{vertical}
There exists a family of isomorphisms $\mathcal{F}:F_\phi^*\to \ell^\phi(X_*)$ such that $\mathcal{F}\circ d''=\delta \circ\mathcal{F}$.
\end{lemma}

\begin{proof}
If $\omega\in F_\phi^m$, then for every $\Delta\in X_m$ the $k$-form $\omega_\Delta$ is an essentially constant function. We define $\mathcal{F} \omega: X_m\to \R$ such that $\mathcal{F} \omega (\Delta)$ is the essential value of $\omega_\Delta$. The inverse of $\mathcal{F}$ is given by $(\mathcal{F}^{-1} \theta)_\Delta\equiv \theta(\Delta)$. It is easy to verify that $\mathcal{F}\circ d''=\delta \circ\mathcal{F}$.

To see that $\mathcal{F}$ and $\mathcal{F}^{-1}$ are well-defined and continuous embeddings suppose that $V\geq 1$ is such that $V^{-1}\leq \Vol(U_\Delta)\leq V$ for every simplex $\Delta$. Then, using Jensen's inequality, we have for every $\alpha$,
\begin{align*}
\sum_{\Delta\in X_m} \phi\left(\frac{|\mathcal{F} \omega(\Delta)|}{\alpha}\right) &= \sum_{\Delta\in X_m} \phi\left(\int_{U_\Delta}\frac{|\omega_\Delta|}{\Vol(U_\Delta)\alpha}\right)dV\\
&\leq V \sum_{\Delta\in X_m}\int_{U_\Delta} \phi\left(\frac{|\omega_\Delta|}{\alpha}\right) dV\\
&\leq V^2 \sum_{\Delta\in X_m} \phi\left(\frac{|\mathcal{F} \omega(\Delta)|}{\alpha}\right).
\end{align*}
This shows that $\|\mathcal{F} \omega\|_{\ell^\phi}\asymp \|\omega\|_{C_\phi}=|\omega|_{C_\phi}$.  
\end{proof}

\begin{proof}[Proof of Theorem \ref{main}]
Combining Theorem \ref{bicomplejo} with Lemmas \ref{retraccion1} and \ref{retraccion2} we have that the complexes $(E_\phi^*,d')$ and $(F_\phi^*,d'')$ are homotopically equivalent. Then, aplying Lemmas \ref{horizontal} and \ref{vertical}, we conclude that $(L^\phi C^*(M),d)$ and $(\ell^\phi (X_*),d)$ are homotopically equivalent, which implies the existence of a family of isomorphisms between the corresponding cohomology spaces.  
\end{proof}

Observe that the previous argument can be done considering $\Omega_\phi^{k,m}$ instead of $C_\phi^{k,m}$, which is the space of alternating functions
\begin{equation}
\omega=\prod_{\Delta\in X_m} \omega_\Delta,\ \omega_\Delta\in L^\phi \Omega^k(U_\Delta)
\end{equation}
with $\|\omega\|<+\infty$. Using the above arguments we can obtain the following result:

\begin{theorem}\label{mainDif}
Take $M$ and $X$ as in Theorem \ref{main} and suppose in addition that
\begin{enumerate}
    \item[(e)] $X$ is a smooth triangulation, and
    \item [(f)] the biLipschitz homeomorphisms given by condition $(b)$ are indeed diffeomorphisms.  
\end{enumerate}
Then for any Young function $\phi$ and $k=0,\ldots,n$, there exists an isomorphism between $L^\phi H_s^k(M)$ and $\ell^\phi H^k(x)$. In particular $L^\phi H_s^k(M)$ is isomorphic to $L^\phi H^k(M)$.
\end{theorem}

If $M$ is a compact Riemannian manifold, then it satisfies the conditions of Theorem \ref{mainDif}. In this case it is easy to see that $L^\phi\Omega^*(M)=\Omega^k(M)$, thus the $L^\phi$-cohomology of $M$ is coincides with its classic de Rham cohomolgy (in the sense of vector spaces).

\section*{Acknowledgments}

I am deeply grateful to Yaroslav Kopylov for his helpful ideas and discussions.  

This work was supported by the \textit{Mathematical Center in Akademgorodok} under the agreement No. 075-15-2019-1675 with the \textit{Ministry of Science and Higher Education} of the Russian Federation.

\medskip
\medskip

\address

\end{document}